\documentclass[reqno,12pt]{amsart}
\usepackage{amsmath,amsthm,enumerate,graphics,pstricks,amsfonts,latexsym,amsopn,verbatim,amscd,amssymb,url}
\usepackage{color}
\usepackage{psfrag}
\usepackage[all]{xy}
\usepackage{tikz}

\theoremstyle{plain}
\newtheorem{thm}{Theorem}[section]

\newtheorem{lem}[thm]{Lemma}
\newtheorem{cor}[thm]{Corollary}
\newtheorem{prop}[thm]{Proposition}
\theoremstyle{definition}
\newtheorem{conj}{Conjecture}[section]
 
\newtheorem{prob}[conj]{Problem}
\newtheorem{rem}[conj]{Remark}

\DeclareMathOperator{\Spec}{Spec}
\DeclareMathOperator{\L-Spec}{L-spec}
\DeclareMathOperator{\Q-Spec}{Q-spec}

\newcommand{\eng}{\mathop{\mathcal{E}}}
\newcommand{\deng}{\mathop{\overrightarrow{\mathcal{E}}}}
\newcommand{\coengall}{\mathop{\mathcal{E}_c}}
\newcommand{\coeng}{\mathop{\mathcal{E}_c^-}}

\begin{document}
\title[Engel and co-Engel graphs]{Engel and Co-Engel graphs of finite groups}

\author[P. J. Cameron, R. Chakraborty,  R. K. Nath and D. Nongsiang]{Peter J. Cameron*, Rishabh Chakraborty, Rajat Kanti Nath and Deiborlang Nongsiang}

\address{P. J. Cameron, School of Mathematics and Statistics, University of St Andrews, North Haugh, St Andrews, Fife, KY16 9SS, UK.}

\email{pjc20@st-andrews.ac.uk}

\address{R. Chakraborty, Department of Mathematical Sciences, Tezpur University, Napaam-784028, Sonitpur, Assam, India. \,\,
Department of Mathematics, Assam Engineering Institute, Chandmari, Guwahati-781003, Assam, India.}

\email{rchakraborty2101@gmail.com}

\address{R. K. Nath, Department of Mathematical Sciences, Tezpur University, Napaam-784028, Sonitpur, Assam, India.}

\email{rajatkantinath@yahoo.com}

\address{D. Nongsiang, Department of Mathematics, North-Eastern Hill University,
	Shillong-793022, Meghalaya, India.}

\email{ndeiborlang@yahoo.in} 

\begin{abstract}
Let $G$ be a group. Associate a directed graph $\deng(G)$ (called the Engel digraph of $G$) with $G$ whose vertex set is $G$, with an arc $(x,y)$ if  
$[y, {}_k x]=1$ for some positive integer $k$,
where $[y,{}_kx]$ is the iterated commutator $[y,x,x,\ldots,x]$, with $k$
terms $x$ in the expression. From this we define the Engel graph $\eng(G)$ by
ignoring directions; the co-Engel graph $\coengall(G)$ is its complement.

The co-Engel graph, under the name ``Engel graph'', was introduced by Abdollahi. However, the name we use is more natural. We begin with some
general results about the Engel digraph and graph,  before turning our attention
to the co-Engel graph. Among other things, we show that (unlike what happens
for the power graph) the undirected Engel graph does not determine the directed
version up to isomorphism, though counterexamples seem to be fairly rare: there
are just two orders less than $100$ for which this happens. We also prove a
universality theorem: every finite digraph is an induced sub-digraph of the
Engel digraph of a finite group. Indeed, the digraph can carry numbers on the
arcs representing the depths of the Engel commutators required; a simple
necessary condition on these numbers is shown to be sufficient.

The isolated vertices of $\coengall(G)$ form the set $F(G)$ be the set of all left
Engel elements of $G$. In a finite group $G$, $F(G)$ is the Fitting subgroup
of $G$ (a result of Abdollahi).
In this paper, we realize the induced subgraph of  co-Engel graphs of certain finite non-Engel groups $G$ induced by $G \setminus F(G)$ (A group $G$ is an
Engel group if for all $x,y\in G$ there exists a positive integer $k$ such that
$[x,{}_ky]=1$. Zorn showed that finite Engel groups are nilpotent.) 
We write $\coeng(G)$  to denote the  subgraph of  $\coengall(G)$ induced by
$G \setminus F(G)$. We also compute genus, various spectra, energies and
Zagreb indices of $\coeng(G)$
for those groups. As a consequence, we determine (up to isomorphism) all finite non-Engel group $G$ such that the clique number $\omega(\coeng(G))$ is at most
$4$ and $\coeng(G)$ is toroidal or projective. Further, we show that
$\coeng(G)$ is ALQ-integral and satisfies the E-LE conjecture and the Hansen-Vuki{\v{c}}evi{\'c} conjecture for the groups considered in this paper. 
\end{abstract}

\thanks{*Corresponding author}

\maketitle

\section{Introduction}

The term ``graphs defined on groups'' refers to graphs whose vertex set is the
group $G$ (or a suitable subset) with the adjacency of elements $x$ and $y$
defined in terms of group-theoretic properties of these elements. The first
such graph was the \emph{commuting graph} of a group, with $x$ and $y$ joined
if $xy=yx$. This was used by Brauer and Fowler in a 1955 paper~\cite{BF} 
containing a foundational result for the study of finite simple groups (they
used the ideas of graph theory but not the name). The graphs defined here can
be regarded as a generalisation of the commuting graph. In the last few decades,
many more such graphs have been defined, including the power graph and
enhanced power graph, nilpotency graph, generating graph, independence graph,
and rank graph. Our topic here is the \emph{Engel graph}.

In a finite group $G$, the \emph{commutator} of two elements $x,y$ is given by
\[[x,y]=x^{-1}y^{-1}xy,\]
and for a positive integer $k$, the $k$th \emph{Engel commutator} is defined
inductively by
\[[x,{}_1y]=[x,y],\qquad[x,{}_ky]=[[x,{}_{k-1}y],y].\]
For $k\ge1$, we say that $(x,y)$ satisfies the $k$th Engel condition if $[x,{}_ky]=1$.

Abdollahi \cite{aa} defined the ``Engel graph'' of a group $G$, 
with $x,y$ joined if they don't satisfy the Engel condition in either order,
and the \emph{left Engel elements} (isolated in this graph) deleted. He showed
that for infinite soluble non-Engel groups, this graph has no infinite clique
if and only if there is a finite bound on its clique number (this is the
analogue of a similar result by Neumann~\cite{neumann} for the non-commuting
graph, answering a question of Erd\H{o}s, and depends on a result of Longobardi
and Maj~\cite{LM}). He also showed that a finite group is soluble if the graph
has clique number at most $15$. This is best possible: the clique number of the co-Engel graph of $A_5$ is 16 (a maximum clique consists of elements of orders 3 and 5, one from each cyclic subgroup).  Moreover, the index of the Fitting subgroup is
bounded by a function of the clique number. Finally, he studied the
connectedness of the graph.

The study languished for some time until it was revived in the present decade.
Recent authors have changed the definition so that Abdollahi's graph is the
complement of the Engel graph as now defined. (Abdollahi gave the first author
of this paper his blessing for this change of usage in \cite{Cameron2022}.)
Also, unlike most graphs defined on groups (the enhanced power graph, commuting
graph, nilpotent graph, generating graph, independence graph, and others), but
similar to the situation with the power graph, the Engel graph arises naturally
as a directed graph, with an arc $x\to y$ if $[x,{}_ky]=1$ for some integer $k$.
The undirected version is obtained by ignoring directions (and regarding two
oppositely directed arcs between the same pair of vertices as a single
undirected edge).

Detomi, Lucchini and Nemmi~\cite{DLN} showed that the Engel digraph of a finite
group $G$ is weakly connected (with undirected diameter at most~$10$), and is
strongly corrected if $G/Z_\infty(G)$ is neither almost simple nor a Frobenius
group. Moreover, if $G$ is soluble and $G/Z_\infty(G)$ is not Frobenius, the
diameter is at most $4$. Dalla Volta, Mastrogiacomo and Spiga~\cite{DVMS}
consider the case where $G$ is almost simple and show that, in most cases,
strong connectedness holds here too. Finally, the first
author~\cite{Cameron2022} used the Engel graph to show that the set of
dominating vertices in the nilpotency graph of a finite group $G$ is the
hypercenter $Z_\infty(G)$.


In this paper, we return to the study of Abdollahi's graph, which we re-name as the co-Engel graph. In Section 2, we recall the necessary definitions and results.
In Section 3, we show that the
conditions on a finite group that its nilpotency graph is complete, its
Engel graph is complete, and these two graphs are equal are all equivalent
to the statement that the group is nilpotent. We  show that the Engel
graph of a group with hypercenter $Z_\infty(G)$ is the lexicographic product
of a complete graph of order $|Z_\infty(G)|$ by the Engel graph of
$G/Z_\infty(G)$. We also show that  the undirected Engel graph does not determine the directed
version up to isomorphism, though counterexamples are not so common. In Section 4, observe that a group is nilpotent if and only if its Engel graph and strong Engel graph are equal, and is two-sided Engel tame if and only if its strong Engel graph is
equal to its nilpotency graph.  In Section 5, we show that every finite digraph is an induced sub-digraph of the Engel digraph of a finite group. Indeed, the digraph can carry numbers on the
arcs representing the depths of the Engel commutators required; a simple
necessary condition on these numbers is shown to be sufficient. In Section 6,  we realize the induced subgraph $\coeng(G)$  of  co-Engel graphs of certain finite non-Engel groups $G$ induced by $G \setminus F(G)$, where $F(G)$ is the Fitting subgroup of $G$. In Section 7, we compute genus of $\coeng(G)$
for those groups. As a consequence, we determine (up to isomorphism) all finite non-Engel group $G$ such that the clique number $\omega(\coeng(G))$ is at most
$4$ and $\coeng(G)$ is toroidal or projective. In Sections 8--9, we compute various  spectra, energies and Zagreb indices of $\coeng(G)$ for a selection of
groups, showing that the graph is ALQ-integral and satisfies the E-LE
conjecture and the Hansen-Vuki\v{c}evi\'c conjecture. In Section 10, we conclude the paper with some problem for future research.

The co-Engel graph is null precisely for nilpotent groups,
so we exclude these from our discussion and remove the isolated vertices of
the co-Engel graph in order to study its properties.

\section{Engel elements}

The \emph{commutator} of two elements $x,y$ of a group $G$ is 
$[x,y]=x^{-1}y^{-1}xy$. It is equal to the identity if the elements commute.

For a positive integer $k$, the $k$th \emph{Engel commutator} of $x$ and $y$ is
defined inductively by
\[[x,{}_1y]=[x,y],\qquad[x,{}_{k+1}y]=[[x,{}_ky],y]\hbox{ for $k\ge1$}.\]
The group $G$ is a \emph{$k$-Engel group}, or satisfies the \emph{$k$-Engel
identity}, if $[x,{}_ky]=1$ for all $x,y\in G$. So $G$ satisfies the $1$-Engel
identity if and only if it is abelian.

The group $G$ is \emph{nilpotent of class~$k$} if every commutator of length
$k+1$ is equal to the identity. So a nilpotent group of class $k$ is a
$k$-Engel group. The converse is false; but Zorn~\cite{zorn} showed that a
finite $k$-Engel group is nilpotent (though its class may be larger than $k$).
So a finite group $G$ is an Engel group (that is, $k$-Engel for some $k$) if
and only if it is nilpotent.

We need one important result, due to Schmidt~\cite{schmidt} (an accessible
account can be found in \cite{ber}). Schmidt gave a complete description of
minimal non-nilpotent finite groups (that is, non-nilpotent groups all of whose
proper subgroups are nilpotent). We need this later; now, we need only the following consequences:

\begin{thm}
A minimal non-nilpotent finite group is soluble.
\end{thm}

\begin{cor}
Any non-nilpotent group contains a soluble non-nilpotent subgroup.
\label{c:nn}
\end{cor}

We also need the following results. The \emph{Fitting subgroup} $F(G)$ of the
finite group $G$ is the (unique) maximal normal nilpotent subgroup of $G$;
the \emph{center} of $G$ is the subgroup
$Z(G)=\{x\in G:(\forall y\in G)([x,y]=1)\}$, and the \emph{hypercenter} of $G$
is the subgroup $Z_\infty(G)$ defined as follows: put $Z_1(G)=Z(G)$; then for
$k\ge1$ let $Z_{k+1}(G)$ be the subgroup defined by
\[Z_{k+1}(G)/Z_k(G)=Z(G/Z_k(G));\]
and then $Z_\infty(G)=\bigcup_{k\in\mathbb{N}}Z_k(G)$. Noting that for a
finite group the ascending series terminates, we see that $Z_\infty(G)=Z_k(G)$
for some $k$. So the hypercenter is nilpotent and normal, and thus is
contained in the Fitting subgroup.

An element $x\in G$ is a \emph{left Engel element} if, for all
$y\in G$, there exists $k$ such that $[y,{}_kx]=1$; $y$ is a \emph{right Engel
element} if, for all $x\in G$, there exists $k$ such that $[y,{}_kx]=1$.

\begin{thm}
Let $G$ be a finite group.
\begin{enumerate}
\item The set of elements $x\in G$ such that $\langle x,y\rangle$ is nilpotent
for all $y\in G$ is the hypercenter of $G$.
\item The set of left Engel elements is the Fitting subgroup of $G$.
\item The set of right Engel elements is the hypercenter of $G$.
\item The set of elements $x\in G$ such that, for all $y$ in $G$, there
exists $k$ such that either $[x,{}_ky]=1$ or $[y,{}_kx]=1$ is the Fitting
subgroup of $G$.
\end{enumerate}
\label{t:engelelts}
\end{thm}

The second and third parts of this theorem are due to Baer~\cite{baer}, and
the last to Abdollahi~\cite{aa}. For the reader's convenience, we
sketch the proof of part (b). Let $F(G)$ be the Fitting subgroup of $G$.
If $x\in F(G)$ and $y\in G$, then $[y,x]\in F(G)$ since $F(G)$ is normal, and
then $[[y,x],{}_kx]=1$ for some $k$, since $F(G)$ is nilpotent. The first 
part is \cite[Theorem 11.4]{Cameron2022}.

\section{Engel digraph and graph}

Let $G$ be a finite group.
\begin{enumerate}
\item The \emph{directed Engel graph}, or \emph{Engel digraph}, of $G$ is
the directed graph which has an arc $x\to y$ if there exists $k$ such that
$[y,{}_kx]=1$. It is denoted $\deng(G)$.
\item The (undirected) \emph{Engel graph} is obtained from the Engel digraph
by ignoring directions; in other words, $x$ and $y$ are joined if there exists
$k$ such that either $[x,{}_ky]=1$ or $[y,{}_kx]=1$. It is denoted
$\eng(G)$.
\item The \emph{co-Engel graph} is the complement of the Engel graph; that is,
$x$ and $y$ are joined if $[x,{}_ky]\ne1$ and $[y,{}_kx]\ne$ for all $k$.
We denote it by $\coengall(G)$, and the result of deleting the isolated
vertices by $\coeng(G)$.
\end{enumerate}

The following is immediate from Theorem~\ref{t:engelelts}:

\begin{thm} Let $G$ be a finite group.
\begin{enumerate}
\item In the Engel digraph, the set of sources (vertices with arcs to all
others) is the Fitting subgroup, while the set of sinks (vertices with arcs
from all others) is the hypercenter.
\item The set of dominating vertices in the Engel graph is the Fitting
subgroup.
\item The set of isolated vertices in the co-Engel graph is the Fitting
subgroup.
\end{enumerate}
\end{thm}

When we consider the co-Engel graph in detail, we will delete the isolated
vertices.

If $G$ is nilpotent, then the Engel digraph is the complete digraph (with arcs
in both directions between all pairs of vertices), and so the Engel graph
is the complete graph. We will extend this observation.

\begin{thm}
Let $G$ be a non-nilpotent finite group. Then there are elements $x,y\in G$
such that there is only a single arc between $x$ and $y$ in the Engel digraph.
\label{t:single}
\end{thm}

\begin{proof}
By Corollary~\ref{c:nn}, we may assume that $G$ is soluble. (A single arc in
the Engel digraph of a subgroup of $G$ is a single arc in the Engel digraph 
of $G$.)

Let $F$ be the Fitting subgroup of $G$. Then $F$ is the set of left Engel
elements, so we have an arc $x\to y$ for all $x\in F$, $y\in G$ with $x\ne y$.
If there are no single edges, then we also have an arc $y\to x$ for all such
$x,y$.

If $F=G$, we are done; so suppose not. Let $H$ be a subgroup of $G$ such that
$H/F$ is a minimal normal subgroup of $G/F$. Then $H/F$ is an elementary
abelian $p$-group, for some prime $p$. We will show that $H$ is nilpotent,
contradicting the fact that $F$ is the maximal normal nilpotent subgroup of $G$.

Take an element $x\in H\setminus F$. We can assume that the order of $x$ is
a power of $p$. Let $Q$ be a non-trivial Sylow $q$-subgroup of $F$, where
$q\ne p$. If $x$ acts non-trivially on $Q$, then it acts non-trivially on
$Q/\Phi(Q)$, by the Burnside basis theorem~\cite[Chapter 10]{hall}, where
$\Phi(Q)$ is the Frattini subgroup of $Q$. This quotient is an elementary
abelian $q$-group, and is completely reducible as an $\langle x\rangle$-module,
since $q\ne p$.

Choose an element $y\in Q$ such that $y\Phi(Q)$ lies in a submodule
$M$ of $Q/\Phi(Q)$ on which $x$ acts non-trivially. Then $x^{-1}yx\Phi(Q)\in M$
and $x^{-1}yx\Phi(Q)\ne y\Phi(Q)$, so $[y,x]\Phi(Q)$ is a non-zero element of
$M$. Now the same argument applies, with induction, to show that
$[y,{}_nx]\Phi(Q)\ne\Phi(Q)$, and hence $[y,{}_nx]\ne1$, for all $n$. But this
contradicts the fact that $x\to y$ in the directed Engel graph.

This shows that a Sylow $p$-subgroup $P$ of $H$ acts trivially on each Sylow
$q$-subgroup $Q$ of $F$ (and hence of $H$) for $q\ne p$ (for elements of
$P\cap F$ commute with $Q$ since $F$ is nilpotent, and the remaining elements
by the argument we have just given). Since $F$ is nilpotent,
its distinct Sylow subgroups commute with each other; so all the distinct
Sylow subgroups of $H$ commute with each other, and $H$ is nilpotent.

This contradiction completes the proof.
\end{proof}

For the next result, we recall that the \emph{nilpotency graph} of the finite
group $G$ has an edge $\{x,y\}$ if and only if $\langle x,y\rangle$ is
nilpotent.

\begin{thm}\label{ad}
Let $G$ be a finite soluble group. Then the following are equivalent:
\begin{enumerate}
\item $G$ is nilpotent;
\item the nilpotency graph of $G$ is complete;
\item the Engel graph of $G$ is complete;
\item the nilpotency and Engel graphs of $G$ are equal.
\end{enumerate}
\end{thm}

This theorem answers part of Question~24 in \cite{Cameron2022}.

\begin{proof}
If $G$ is nilpotent, then its $2$-generator subgroups are nilpotent and hence
Engel, so (b)--(d) all hold. We prove the converses.

Suppose that $G$ is not nilpotent. By Theorem~\ref{t:single}, the Engel digraph
has two vertices $x$ and $y$ joined by a single arc. Then one of
$[x,{}_ky]=1$ and $[y,{}_kx]=1$ fails for all $k$. Thus, $\langle x,y\rangle$ is
not Engel, and hence not nilpotent, so the nilpotency graph is not complete.
Thus (b) implies (a). Moreover, $\{x,y\}$ is an edge in the Engel graph but
not in the nilpotency graph; so also (d) implies (a).

Using Schmidt's classification of
the minimal non-nilpotent groups, we see that each of them involves a cyclic
$q$-group $Q$ acting fixed-point-freely on an abelian $p$-group $P$. Let $x$
be a generator of $Q$, and $y$ a conjugate of $x$ by an element of $P$.
Then $[x,y]$ is a non-identity element of $P$, and so $[x,{}_ky]$ is a
non-identity element of $P$ for all $k$. Similarly with $x$ and $y$
interchanged. So $x$ and $y$ are not joined in the Engel graph of $G$. Thus
(c) implies (a).
\end{proof}

Next we give a structure theorem for the Engel graph which shows that, for many
purposes, we can reduce to the case of groups with trivial center. We note
that, for any finite group $G$, the quotient $G/Z_\infty(G)$ has trivial center.

The \emph{lexicographic product} of graphs $\Gamma_1$ and
$\Gamma_2$ is the graph whose vertex set is $v(\Gamma_1)\times v(\Gamma_2)$,
which we regard as a disjoint union of copies $C_x$ of $v(\Gamma_1)$ indexed by
$\Gamma_2$. We place a copy of $\Gamma_1$ on each of these sets. For $x\ne y$
if there is an edge (resp. arc) from $x$ to $y$ in $\Gamma_2$, then we put
all edges (resp. arcs) from $C_x$ to $C_y$; otherwise we put no edges.

\begin{thm}
Let $G$ be a finite group.
\begin{enumerate}
\item The Engel digraph of $G$ is isomorphic to the lexicographic product of
a complete directed graph of order $|Z_\infty(G)|$ by the Engel digraph of
$G/Z_\infty(G)$.
\item The Engel graph of $G$ is isomorphic to the lexicographic product of
a complete graph of order $|Z_\infty(G)|$ by the Engel graph of
$G/Z_\infty(G)$.
\item The co-Engel graph of $G$ is isomorphic to the lexicographic product of
a null graph of order $|Z_\infty(G)|$ by the co-Engel graph of $G/Z_\infty(G)$.
\end{enumerate}
\end{thm}

\begin{proof}
Parts (b) and (c) follow immediately from (a), so it suffices to prove (a).
We prove this by induction on the length of the upper central series 
$(Z_n(G))$ of $G$. We have to show that, if some pair of vertices in
$Z_\infty(G)$ are joined, then all pairs in the same cosets are joined.

If the length is $1$, then $Z_\infty(G)=Z(G)=Z$, say, the center of $G$.
The elements of $G/Z$ are the cosets of $Z$ in $G$. If $[x,{}_ky]=1$, then
$[xZ,{}_kyZ]=Z$. Conversely, if $[xZ,{}_kyZ]=Z$ in $G/Z$, then $[x,{}_ky]\in Z$
in $G$, and so $[x,{}_{k+1}y]=1$. Moreover, the induced subgraph on $Z$ is
complete. This demonstrates the result in this case.

Suppose that $n\ge1$ and we have proved the result when $Z_\infty(G)=Z_n(G)$.
For $G/Z_{n+1}(G)$ with $n>1$, we have
\[G/Z_n(G)\cong(G/Z_{n-1}(G))/(Z_n(G)/Z_{n-1}(G)),\]
and the result follows.
\end{proof}

We will use this theorem a number of times in what follows. It will also be
useful in further research: it implies that, in studying groups $G$ whose
co-Engel graph is perfect, or chordal, or a cograph, or in investigating the
effect of twin reduction, we may assume without loss that the center of $G$
is trivial.

Other results about the Engel digraph do not extend to the undirected case.
For example, the \emph{strong product} of digraphs $D_1$ and $D_2$ has
vertex set $V(D_1)\times V(D_1)$, with an arc $(a,x)\to(b,y)$ if either
$a=b$ or $a\to b$, and either $x=y$ or $x\to y$, but not both $a=b$ and
$x=y$.

\begin{prop}
Let $G$ and $H$ be finite groups. Then the Engel digraph of $G\times H$
is the strong product of the Engel digraphs of $G$ and $H$.
\end{prop}

The proof is straightforward. However, it does not hold for the Engel and
co-Engel graphs. For, if there are single arcs $a\to b$ and $y\to x$, then
the strong product of the Engel graphs has an edge $\{(a,x),(b,y)\}$ which
is not an edge of the Engel graph of $G\times H$. (This problem does not arise
if one of the factors is nilpotent, since its Engel digraph has no single
arcs. Indeed, if $G$ is nilpotent, then it is contained in the hypercenter
of $G\times H$.)

\medskip

We mention here a result of Abdollahi~\cite{aa1}:

\begin{thm}
For a finite group $G$, the index $|G:F(G)|$ is bounded by a function of the
clique number of the co-Engel graph of $G$ (that is, the independence number
of the Engel graph).
\end{thm}

It follows, for example, that the number and orders of non-abelian composition
factors of $G$ are bounded by a function of the clique number of the
co-Engel graph; indeed, he showed that, if the clique number is at most~$15$,
then the group is soluble.

An early result on the power graph of a finite group~\cite{Cameron:pg}  is
that, if the power graphs of two groups are isomorphic, then the directed power
graphs are isomorphic. We pose the same question for the Engel graph.
Computationally, we found just two orders less than $100$ for which this is
false, namely $54$ and $96$. For $n=54$, there are five groups (numbers $3$,
$5$, $6$, $8$ and $13$ in the list of small groups of order $54$ in
\textsf{GAP}~\cite{gap}), all of which have isomorphic undirected Engel graphs;
 but the directed Engel graphs are of two isomorphism types, one realised by
groups number $5$ and $6$, the other by the other three. Similarly, for $n=96$,
 groups with numbers $3$, $68$, $70$, $71$, $72$, $203$, $204$ and $229$ all
have isomorphic undirected Engel graphs, but the directed graphs fall into two
isomorphism types, one for groups $70$, $71$ and $72$, the other for the rest.

\section{Nilpotence, Fitting subgroup, and Engel digraph}
\label{s:NFE}

It is clear from the definition that, if $H$ is a subgroup of $G$, then the
induced subgraph on $H$ of the Engel digraph of $G$ is the Engel digraph of
$H$. It follows that whether two elements $x$ and $y$ are joined by a double
arc, a single arc, or none at all is determined by the subgroup
$\langle x,y\rangle$. Moreover, we have seen the following:
\begin{enumerate}
\item If $\langle x,y\rangle$ is nilpotent, then there are arcs $x\to y$
and $y\to x$.
\item If $y$ is in the Fitting subgroup of $\langle x,y\rangle$, then there
is an arc $y\to x$.
\end{enumerate}
Things would be much simpler if the converses of these statements were true;
but they are not.
\begin{enumerate}
\item If $x=(1,2,3)(4,5)$ and $y=(1,4,5,8)(2,6)(3,7)$, then $[x,{}_4y]=1$ and
$[y,{}_3x]=1$, but $\langle x,y\rangle=S_8$ (which is not nilpotent).
\item If $x=(1,2)(3,4,5)$ and $y=(2,3)$, then $[x,{}_2y]=1$, but 
$\langle x,y\rangle=S_5$ (whose Fitting subgroup is trivial).
\end{enumerate}

But computations with \textsf{GAP}~\cite{gap} show that the groups $S_4$ and
$A_5$ do satisfy the converse of (b), while $S_7$ does satisfy the converse
of (a).

Accordingly, we make the following definitions:
\begin{enumerate}
\item A group $G$ is \emph{two-sided Engel tame} if for all $x,y\in G$, if
$x\to y$ and $y\to x$ in $\deng(G)$, then $\langle x,y\rangle$ is nilpotent.
\item A group $G$ is \emph{one-sided Engel tame} if for all $x,y\in G$, if
$y\to x$ in $\deng(G)$, then $y\in F(\langle x,y\rangle)$.
\end{enumerate}
Note that the classes of one- and two-sided Engel tame groups are both
closed under taking subgroups, finite direct products, and central quotients
and extensions. The proofs of these are all very similar. As an illustration,
here is the proof that a central quotient of a one-sided Engel tame group has
the same property.

Suppose that $G$ is one-sided Engel tame and let $Z$ be a subgroup of $Z(G)$.
Take $x,y\in G$ and suppose that $[xZ,{}_kyZ]=Z$ in $G/Z$. Then
$[x,{}_ky]\in Z$, so $[x,{}_{k+1}y]=1$. By assumption,
$y\in F(\langle x,y\rangle)$. Then
$yZ\in F(\langle x,y\rangle)Z/Z$, which is a nilpotent normal subgroup of
$\langle xZ,yZ\rangle/Z$, so $yZ\in F(\langle xZ,yZ\rangle/Z)$, completing
the proof.

It may also be that investigations such as those in the remainder of this paper
may  be easier for groups in one or other of these classes. Note, however, that
the counterexamples to reconstruction of the directed Engel graph from
the undirected version, at the end of the preceding section, are all both
one- and  two-sided Engel tame.

Another approach to the study of pairs $(x,y)$ with $x\to y$ and $y\to x$ is
as follows. We define the \emph{strong Engel graph} of $G$ to be the graph
with vertex set $G$, having an edge $\{x,y\}$ if $x\to y$ and $y\to x$.

The edge set of this graph contains the edge set of the nilpotency graph,
the graph with edges $\{x,y\}$ if $\langle x,y\rangle$ is nilpotent, and is
contained in the edge set of the Engel graph. By Theorem~\ref{t:single}, the
strong Engel graph of $G$ is equal to the Engel graph if and only if $G$
is nilpotent. Also, it is equal to the nilpotency graph if and only if $G$ is
two-sided Engel tame.

We have not studied the strong Engel graph, but can report the following
result using the notion of Engel depth. First, a few definitions.
The \emph{Engel depth} of the pair $(x,y)$ is the minimum $k$ such that
$[x,{}_ky]=1$, if such $k$ exists; otherwise the Engel depth is infinite. Now
we can define the \emph{$k$-Engel digraph} of $G$ to have an arc $x\to y$ if
$(x,y)$ has Engel depth at most $k$. Similarly, the $k$-nilpotency graph has
an edge $\{x,y\}$ if and only if $\langle x,y\rangle$ is nilpotent with class
at most~$k$.

\begin{prop}
Let $x$ and $y$ be elements of a group $G$.
\begin{enumerate}
\item If $(x,y)$ has Engel depth $1$, then so does $(y,x)$, and $\{x,y\}$ is
an edge of the commuting graph of $G$.
\item Both $(x,y)$ and $(y,x)$ have Engel depth $2$ if and only if
$\langle x,y\rangle$ is nilpotent of class~$2$.
\end{enumerate}
\end{prop}

\begin{proof}
(a) is clear, so we prove (b). Suppose that both $(x,y)$ and $(y,x)$ have
Engel depth $2$; then $[x,y,y]=[y,x,x]=1$. Since $[y,x]=[x,y]^{-1}$, this
means that $[x,y]$ commutes with both $x$ and $y$, and so lies in the center
of $H=\langle x,y\rangle$.

Let $Z=Z(H)$. Then $[x,y]\in Z$, so $[xZ,yZ]=1$ in $H/Z$, which implies that
$H/Z$ is abelian; so $H$ is center-by-abelian and so nilpotent of class at
most~$2$. It is non-abelian, so has class~$2$.

Conversely, if $\langle x,y\rangle$ is nilpotent of class at ost~$2$, then
$H'\le Z(H)$, so $[x,y]$ commutes with both $x$ and $y$; so it lies in $Z(H)$,
whence $[x,y,y]=[y,x,x]=1$.
\end{proof}

Thus, the double arcs in the $2$-Engel graph are the edges of the
$2$-nilpotency graph. Examples above show that there is no similar result for
larger Engel depth.

We are grateful to Michael Kinyon for this argument, which simplifies
our original proof.

\section{Universality of Engel digraphs}

Engel digraphs of finite digroups are \emph{universal}: that is, every finite
digraph can be embedded as an induced sub-digraph in the Engel digraph of
some finite group. This contrasts whith the situation for other digraphs on
groups such as the power digraph~\cite{kq} and endomorphism digraph~\cite{acgl},which are transitive as relations.

We are actually going to prove a much stronger theorem, using the notion of
Engel depth defined above. We prove the following theorem:

\begin{thm}
Let $D$ be a finite digraph, whose arcs are labelled with positive integers.
Then the following two conditions are equivalent:
\begin{enumerate}
\item If an arc $x\to y$ has label $1$, then there is also an arc $y\to x$
with label $1$.
\item There is a finite group $G$ and an embedding of the vertex set of
$D$ into $G$ with the properties that $x\to y$ is an arc with label $k$ if
and only if $(x,y)$ has Engel depth $k$, and there is no arc $x\to y$ if and
only if $(x,y)$ has infinite Engel depth.
\end{enumerate}
\end{thm}

\begin{proof}
First we show that (b) implies (a). The pair $(x,y)$ has Engel
depth $1$ if and only if $x$ commutes with $y$; this is equivalent to saying
that $y$ commutes with $x$, so also $(y,x)$ has Engel depth~$1$.

\medskip

To prove the converse, assume that (a) holds; for every possible $2$-vertex
digraph, we are going to construct a group $G$ with distinguished elements
$x$ and $y$ with the properties that there is an arc $x\to y$ with label $k$
if and only if $(x,y)$ has Engel depth $k$ (and no arc if and only if the
Engel depth is infinite), and similarly with $x$ and $y$ interchanged.

First we make an observation about direct products. Suppose that
$x_i,y_i\in G_i$ for $i=1,2$. Then the Engel depth of $(x_1x_2,y_1y_2)$
in $G_1\times G_2$ is the maximum of the Engel depths of $(x_1,y_1)$ in $G_1$
and $(x_2,y_2)\in G_2$. The proof is clear.

Consider the dihedral group 
\[G=\langle a,b\mid a^{2^n}=1,b^2=1,b^{-1}ab=a^{-1}\rangle\]
of order $2^{n+1}$. Now $[b,a]\in A$, so $[b,a,a]=1$; this $(a,b)$ has
Engel depth $2$. On the other hand, $[a,b]=a^{-2}$, whose order is
$2^{n-1}$; and each further commutation with $b$ halves the order, so the
Engel depth of $(b,a)$ is $n$.

Let $b'=ab$. Then $[b,b']=a^{-2}$, with order $2^{n-1}$; as above, we find
that $(b',b)$ has Engel depth $n$. The same holds for $(b,b')$.

So we can realise the case where there are arcs $x\to y$ with label $k$ and
$y\to x$ with label $l$ as follows:
\begin{itemize}
\item If $k=l=1$, let $x$ and $y$ be distinct elements in an abelian group $G$.
\item If $k=l>1$, let $x=b$, $y=ab$ in the dihedral group of order $2^{k+1}$.
\item If $k=2$ and $l>2$, let $x=a$, $y=b$ in the dihedral group of order
$2^{l+1}$.
\item If $2<k<l$, take $G$ to be the direct product of dihedral groups of orders
$2^{k+1}$ (with generators $a_1$ and $b_1$) and $2^{l+1}$ (with generators
$a_2$ and $b_2$), with $x=b_1a_2$ and $y=a_1b_1b_2$.
\end{itemize}

Now suppose that $k$ is finite and greater than $1$, while $l$ is infinite.
Take $G_1$ to be dihedral of order $2^{k+1}$, and $G_2$ to be dihedral of 
order $6$ (with $c$ of order~$3$ and $d$ of order $2$). In $G_1\times G_2$,
take $x=bc$ and $y=abd$, noting that $(c,d)$ has infinite Engel depth.

Finally, suppose that both $k$ and $l$ are infinite. We can take $G$ to be
the direct product of two copies of the dihedral group of order $6$, with 
$x=c_1d_2$ and $y=d_1c_2$.

\medskip

So the theorem is proved for $2$-vertex subgraphs, and we need to extend to
subgraphs with larger numbers of vertices.

For each pair of vertices $u,v$ of $D$, choose a group realising the induced
sub-digraph on $\{u,v\}$ according to the arguments just given. There is no
overlap between groups for different pairs. Let $G(u,v)$ be the group
corresponding to the vertices $u$ and $v$, and let $g_u$ and $g_v$ be the
group elements which are identified with those vertices. (A little care is
required; in some cases the two elements $x$ and $y$ are symmetric in the
group, we have to specify which is associated with which vertex of $D$.)

Take $G$ to be the direct product of these groups over all pairs $u,v$. Now
take a vertex $u$.  We identify it with an element of the direct product as
follows: in any factor indexed by a pair containing $u$, we choose the element
$g_u$; in any other factor, we choose the identity. Clearly different vertices
of $D$ are identified with different group elements. Now when we evaluate an
Engel commutator, say $[y,{}_kx]$, on a pair $x=u$, $y=v$, the commutator
evaluates correctly in the group $G(u,v)$; for any other group, at least one
of $x$ and $y$ has been identified with the identity, and so even the first
commutator is trivial. Thus the Engel depth of the group elements chosen to
represent $u$ and $v$ agrees with the number on the arc $u\to v$ if there is
one, and is infinite if there is no such arc.

Thus, the theorem is proved.
\end{proof}

There is a similar definition of universality for a class of finite graphs. Now
we have the following.

\begin{cor}
The class of co-Engel graphs is universal.
\end{cor}

The proof is almost immediate from the Theorem above.

\section{Realization  of  $\coeng(G)$}\label{s:next}

From this point on, we turn our attention to properties of co-Engel graphs of
some non-Engel (that is, non-nilpotent groups). Our focus will be on particular
examples.  We remind the reader that, from this point on, we delete the isolated vertices (the elements of the Fitting subgroup). 

In this section we realize certain graphs as $\coeng(G)$ for certain finite groups.
Let $K_n$ be the complete graph on $n$ vertices and  $K_{n_1, n_2, \dots, n_m}$ be the complete $m$-partite graph with parts of size $n_1, n_2, \dots, n_m$. We write $K_{m \cdot n} = K_{n_1, n_2, \dots, n_m}$  if $n_1 = n_2= \cdots = n_m = n$.

We denote the dihedral group of order $2n$ by $D_{2n} $ (in the literature this
group is sometimes called $D_n$).

It was shown in \cite{aa} that 
$\coeng(D_6)\cong K_3$ and $\coeng(D_{12})\cong K_{3\cdot2}$.
In the following theorem we shall generalize these results and realize 
$\coeng(G)$, with $G$ is a dihedral group
$D_{2m} = \langle x,y \mid y^m = x^2= 1,xyx^{-1} = y^{-1} \rangle$ or $D_{2^{t+1}m} = \langle x,y \mid y^{2^tm} = x^2= 1,xyx^{-1} = y^{-1}\rangle$, or a generalised quaternion group $Q_{2^{t+1}m} = = \langle x,y \mid y^{2^tm} = 1, x^2=y^{2^{t-1}m},xyx^{-1} = y^{-1}\rangle$,
where $t \geq 1$ and $m \geq 3$ is odd.

\begin{thm} \label{dihed} 
If $t \geq 1$ and $m \geq 3$ is odd, then $\coeng(D_{2^{t+1}m})$
and $\coeng(Q_{2^{t+1}m})$ are isomorphic  to $K_{m \cdot 2^t}$. Further, 
$\coeng(D_{2m})$  is isomorphic to $K_m$. 
\end{thm}
\begin{proof}
Suppose that $G$ denotes one of the groups $D_{2^{t+1}m}$ or $Dic_{2^{t+1}m}$. Then	
 $F(G) = \langle y\rangle$. Therefore, the set of non-isolated vertices is $G\setminus F(G)=\{x,xy,\dots,xy^{2^tm-1}\}$, where $F(G)$ is the Fitting subgroup. We have
\[ 
[xy^i, {}_nxy^j]= \begin{cases} 
      y^{2^n(i-j)}, & \textrm{ if $n$ is even} \\
      y^{2^n(j-i)}, & \textrm{ if $n$ is odd.} \\
   \end{cases} 
\]
Thus, $[xy^i, {}_nxy^j]= 1$ if and only if $2^tm\mid 2^n(i-j)$, that is $i \equiv j \mod m$. Let
$$
A_j=\{xy^i \mid 0\leq i\leq 2^tm-1, i \equiv j \mod m\}
$$
for $j=0, 1, \ldots,m-1$.
Then $A_0, A_1, \dots, A_{m-1}$ are all disjoint subsets of $G\setminus F(G)$ such that $|A_i| = 2^t$ for all $0 \leq i \leq m - 1$. Notice that $u\in A_i$ and $v\in A_j$ are adjacent if and only if $i\neq j$. Therefore, $\coeng(G)\cong K_{2^t, \dots, 2^t}=K_{m \cdot 2^t}$. 

On the other hand, for $D_{2m}$ (where $m$ is odd) we have $|A_i|=1$ for all $0 \leq i \leq m - 1$. Hence, $\coeng(D_{2m})\cong K_m$. 
\end{proof}

In the next theorem, we realize $\coeng(G)$ when $G = F_{p, q}$, the non-abelian group of order $pq$ where $p$ and $q$ are primes.
\begin{thm} \label{pq} 
If $G$ is the group $F_{p, q} = \langle a,b \mid a^p=b^q=1, a^{-1}ba=b^r\rangle$ of order $pq$, where $p$ and $q$ are primes such that $q\equiv 1 \mod p$ and $r^p\equiv 1 \mod q$, then $\coeng(G)\cong K_{q \cdot (p-1)}$. 
\end{thm}

\begin{proof}
We have \quad $F_{p, q} =\{1,b,b^2,\dots,b^{q-1},a,a^2,\dots,a^{p-1},ab,\dots,$ $ab^{q-1},\dots,a^{p-1}b,\dots,$ $a^{p-1}b^{q-1}\}$. Thus, if  $g \in F_{p, q}$ then  $g = a^ib^j$, where $0\leq i\leq p-1$ and $0\leq j\leq q-1$.

We also have $a^ib^ta^jb^s=a^{i+j}b^{s+tr^j}$. Therefore, if $i\neq 0$ then $(a^ib^j)^p=1$. Note that the group $F_{p, q}$ is the union of its subgroups $H_b=\langle b\rangle = F(F_{p, q})$ and 
$$
H_i=\langle ab^i\rangle =\{1,ab^i,a^2b^{i(r+1)},a^3b^{i(r^2+r+1)}, \dots,a^{p-1}b^{i(r^{(p-1)-1}+r^{(p-1)-2}+\dots+r+1)}\},
$$
where $0 \leq i \leq q-1$. It is easy to see that the intersection of any two of these subgroups is $Z(F_{p, q}) = \{1\}$.

Let $x=a^tb^{i(r^{t-1}+\dots+r+1)}=a^tb^{it'}\in H_i$ and $y=a^sb^{j(r^{s-1}+\dots+r+1)}=a^tb^{js'}\in H_j$, where $t'=r^{t-1}+\dots+r+1$ and $s'=r^{s-1}+\dots+r+1$. If $i=j$, then $[x,y]=1$. So suppose that $i\neq j$. We have 
$$
[x,y]=b^{it'(r^s-1)+js'(1-r^t)}=b^k,
$$
where $k={it'(r^s-1)+js'(1-r^t)}$. Thus
$$
[x, {}_2y]=[b^k,a^sb^{js'}]=b^{k(r^s-1)}=b^l,
$$
where $l=k(r^s-1)$. Therefore. $$
[x, {}_3y]=[b^l,a^sb^{js'}]=b^{l(r^s-1)}=b^{k(r^s-1)^2}.
$$
In general, we have $[x, {}_ny]=b^{k(r^s-1)^{n-1}}$.

Again, 
\begin{align*}
k &=it'(r^s-1)+js'(1-r^t)\\
&=i(t'r^s-t')+j(s'-r^ts')=i(t'r^s+s'-s'-t')+j(s'-r^ts'-t'+t')\\
&=i((t+s)'-s'-t')+j(s'- (t+s)' + t'),
\end{align*} 
where \quad $(t+s)' = r^{t+s-1}+\dots+r^{s+1}+r^s+r^{s-1}+\dots+r+1$. Thus 
\begin{align*}
k=(i-j)((t+s)'-s'-t') &=(i-j)\left(\frac{r^{t+s}-1}{r-1}-\frac{r^{s}-1}{r-1}-\frac{r^{t}-1}{r-1}\right)\\
& = \frac{(i-j)(r^{t}-1)(r^{s}-1)}{r-1}.
\end{align*}
Clearly $q$ does not divide $(i-j), (r^{t}-1)$ and $(r^{s}-1)$, since $1\leq i, j \leq q-1$, $i\neq j$ and $0\leq r, s \leq p-1$.
Therefore, $q \nmid k(1-r^s)^{n-1}$ and so $[x, {}_n y]\neq 1$ for any integer $n$. Hence, the result follows.
\end{proof}

The following remark is useful in realizing $\coeng(G)$ if $G$ is a  direct product of two groups.
\begin{rem}\label{remcr}
\begin{enumerate}
\item Let $G$ and $H$ be two groups and $(u,x),(v,y)\in G\times H$. Then
\[
[(u,x),(v,y)]=(u^{-1}v^{-1}uv,x^{-1}y^{-1}xy)=([u,v],[x,y]).
\]
In general, we have
\[
[(u,x), {}_n[v,y])=([u, {}_nv],[x, {}_ny]).
\]
\item Let $G$ be a group and $x,y \in G$. If $[x, {}_ny]=1$ for some positive integer $n$, then $[x, {}_{n+1}y]=[[x, {}_ny],y]=[1,y]=1$. Thus $[x, {}_my]=1$ for all $m\geq n$.
\end{enumerate}
\end{rem}

\begin{thm}\label{bipar} 
Let $G$ be a finite non-Engel group such that $\coeng(G))\cong K_{m \cdot n}$. Let $H$ be a finite Engel group with $|H|=l$. Then $\coeng(H\times G)\cong K_{lm \cdot n}$. 
\end{thm}
\begin{proof}
Let $(u,x),(v,y)\in H\times G$. Then, by the Remark \ref{remcr}(a), we have
$$
[(u,x), {}_k(v,y)]=([u, {}_kv],[x, {}_ky])\text{ and } [(v,y), {}_k(u,x)]=([v, {}_ku],[y, {}_kx]).
$$
Since $H$ is a finite Engel group, there exist  positive integers $i$ and $j$, such that $[u, {}_iv]=1$ and $[v, {}_ju]=1$. Suppose $[x, {}_ay]=1$ and $[y, {}_bx]=1$ for some positive integers $a$ and $b$. Let $r\geq \max\{i,a\}$ and $s\geq \max\{j,b\}$. Then, by Remark \ref{remcr}(b), we have $([u, {}_rv],[x, {}_ry])=1$ and $([v, {}_su],[y, {}_sx])=1$. Hence, it follows that $([u, {}_kv],[x, {}_ky])\neq 1$ and $([v, {}_ku],[y, {}_kx])\neq 1$ if and only if $[x, {}_ky]\neq 1$ and $[y, {}_kx]\neq 1$ for all positive integer $k$. Thus, $F(H\times G)=H\times F(G)$ and hence $v(\coeng(H\times G))=H\times (G \setminus F(G))$. 

Let $G_1, G_2,\dots,G_m$ be the partite sets of $\coeng(G) \cong K_{m \cdot n}$. Since  $([u, {}_kv],[x, {}_ky])$ $\neq 1$ and $([v, {}_ku],[y, {}_kx])\neq 1$ if and only if $[x, {}_ky]\neq 1$ and $[y, {}_kx]\neq 1$ for all positive integer $k$, we have $\coeng(H\times G)$ is isomorphic to $K_{lm \cdot n}$ with partite sets $H\times G_1, H\times G_2,\dots,H\times G_m$. This completes the proof.
\end{proof}

\section{Genus of $\coeng(G)$}\label{S:Genus}

The genus $\gamma(\Gamma)$ of a graph $\Gamma$ is the smallest genus of a surface in
which it can be drawn without crossing edges.
In this section, we determine the genus of $\coeng(G)$ for the groups considered in  Section~\ref{s:next}. It is well-known that 
\begin{equation}\label{genus-Kn}
\gamma(K_n) = \left\lceil \frac{(n - 3)(n-4)}{12}\right
\rceil
\end{equation}
for $n \geq 3$. 
If $K_{m, n}$ denotes the complete bipartite graph with parts of size  $m, n \geq 2$, then 
\begin{equation}\label{genus-Km,n}
	\gamma(K_{m, n}) = \left\lceil \frac{(m - 2)(n-2)}{4}\right
	\rceil.
\end{equation}
From the main theorem of \cite{whi-1969}, we also have
\begin{equation}\label{genus-Kn,n,n}
\gamma(K_{mn, n, n}) =  \frac{(mn- 2)(n-1)}{2}
\end{equation}
for all positive integers $m$ and $n$.
The following is useful in our computation.
\begin{thm}\label{genus-Kab}
If $\Gamma = K_{a \cdot b}$, where $a\geq 3, b \geq 2$, then $\gamma(\Gamma) = \frac{a(a - 1)}{2}\left\lceil\frac{(b-2)^2}{4}\right
\rceil + \left\lceil \frac{(a-3)(a-4)}{12}\right
\rceil$.
\end{thm}
\begin{proof}
Follows from Theorem 4.1 of \cite{bn-2010}.
\end{proof}

Using Theorem \ref{dihed}--Theorem \ref{bipar} and Theorem \ref{genus-Kab}, we get the following results. 

\begin{thm} \label{dihed-genus} 
If $t \geq 1, m \geq 3$, and $m$ is odd, then the genus of $\coeng(G)$, where $G=D_{2^{t+1}m}$ and  $Q_{2^{t+1}m}$,  
are given by 
\[
 \gamma(\coeng(G)) = \frac{m(m - 1)(2^{t-1}-1)^2}{2} + \left\lceil \frac{(m-3)(m-4)}{12}\right
\rceil.
\]
Further,  $\gamma(\coeng(D_{2m})) = \left\lceil \frac{(m-3)(m-4)}{12}\right
\rceil$.
\end{thm}


\begin{thm} \label{pq-genus} 
If $F_{p, q}=\langle a,b \mid a^p=b^q=1, a^{-1}ba=b^r\rangle$, where $p$ and $q$ are primes, $q\equiv 1 \mod p$ and $r^p\equiv 1 \mod q$ then
\[
\gamma(\coeng(F_{p, q}))= \begin{cases}
\left\lceil \frac{(q-3)(q-4)}{12}\right
\rceil, &\text{ if } p = 2	\vspace{.2cm}\\

\frac{q(q - 1)}{2}\left\lceil\frac{(p-3)^2}{4}\right
\rceil + \left\lceil \frac{(q-3)(q-4)}{12}\right
\rceil, &\text{ if } p \geq 3.
\end{cases}
\]
\end{thm}

%
%

\begin{thm}\label{bipar-genus} 
Let $G$ be a finite non-Engel group such that $\coeng(G) \cong K_{m \cdot n}$. Let $H$ be a finite Engel group with $|H|=l$ and $X = H \times G$. Then  
\[
\gamma(\coeng(X)) =  \frac{lm(lm - 1)}{2}\left\lceil\frac{(n-2)^2}{4}\right
\rceil + \left\lceil \frac{(lm-3)(lm-4)}{12}\right
\rceil.
\]
\end{thm}

 
\begin{cor} 
Let $G$ be a finite non-Engel group such that $\coeng(G) \cong K_{m}$. Let $H$ be a finite Engel group with $|H|=l$  and $X = H \times G$. Then $\gamma(\coeng(X)) = \left\lceil \frac{(lm-3)(lm-4)}{12}\right\rceil.$
\end{cor}

\begin{cor} \label{dihede}
Let $H$ be a finite Engel group. Let $t$ and $m$ be positive integers with $t\geq 1$, $m\geq 3$ and $m$ is odd. Then 
$$
\coeng(H\times D_{2^{t+1}m)} \cong \coeng( H\times Q_{2^{t+1}m}).
$$	
Let $|H| = l$ and $K=D_{2^{t+1}m}$ or $Q_{2^{t+1}m}$. Then $\coeng(H\times K) \cong K_{lm \cdot 2^t}$ and  $\coeng(H\times D_{2m}) \cong K_{lm}$. 
Further,
\[
\gamma(\coeng(H\times K))
= \frac{lm(lm - 1)(2^{t-1}-1)^2}{2} + \left\lceil \frac{(lm-3)(lm-4)}{12}\right
\rceil
\]
and $\gamma(\coeng(H\times D_{2m})) = \left\lceil \frac{(lm-3)(lm-4)}{12}\right\rceil$.
\end{cor}
%

\begin{cor} \label{pqcor} 
Let $F_{p, q} =\langle a,b \mid a^p=b^q=1, a^{-1}ba=b^r\rangle$ be a group of order $pq$, where $p$ and $q$ are primes such that $q\equiv 1 \mod p$ and $r^p\equiv 1 \mod q$. Let $H$ be a finite Engel group with $|H|=l$. Then $\coeng(H\times F_{p, q}) \cong K_{lq \cdot (p-1)}$.
Further, 
\[
\gamma(\coeng(H\times F_{p, q})) =  \frac{lq(lq - 1)}{2}\left\lceil\frac{(p-3)^2}{4}\right
\rceil + \left\lceil \frac{(lq-3)(lq-4)}{12}\right
\rceil. 
\]
\end{cor}
%

The remaining part of this section is devoted to the characterization of finite groups through the genus of $\coeng(G)$. We first recall the following characterization of finite groups obtained by Abdollahi \cite{aa}.
\begin{thm}\label{Ab-Plannar}[\cite{aa}, Theorem 3.1]
Let $G$ be a finite non-Engel group. Then $\coeng(G)$ is planar if and only if $G\cong D_6,D_{12}$ or $Q_{12}$.
\end{thm}

In the next two results  we characterize the groups $G=D_{2^{t+1}m}$, $Q_{2^{t+1}m}$, $D_{2m}$  (where $t \geq 1$, $m \geq 3$ and $m$ is odd) or $F_{p,q}$ such that the  graphs $\coeng(G)$ are planar, toroidal, double-toroidal and triple-toroidal.
\begin{thm}
If $G$  is isomorphic to  the group $D_{2m}$, where $m \geq 3$ and $m$ is odd, then
\begin{enumerate}
\item $\coeng(G)$ is planar if and only if $m = 3$.
\item $\coeng(G)$ is toroidal if and only if $m = 5, 7$.
\item $\coeng(G)$ is not double-toroidal.
\item $\coeng(G)$ is triple-toroidal if and only if $m = 9$.
\item  $\gamma(\coeng(G)) \geq 5$ for $m \geq 11$.
\end{enumerate} 
\end{thm}
\begin{proof}
If $G$ is isomorphic to $D_{2m}$, $m \geq 3$ and $m$ is odd, then
by Theorem \ref{dihed-genus} we have $\gamma(\coeng(G)) = \left\lceil \frac{(m-3)(m-4)}{12}\right
\rceil$. If $m = 3$ then we have $\frac{(m-3)(m-4)}{12} = 0$. Therefore, $\gamma(\coeng(G)) = 0$. If $m= 5$ and $7$ then $\frac{(m-3)(m-4)}{12} = \frac{1}{6}$ and $1$ respectively. Therefore, $\gamma(\coeng(G)) = 1$.  If $m = 9$ then $\frac{(m-3)(m-4)}{12} = 2.5$. Therefore, $\gamma(\coeng(G)) = 3$. If  $m \geq 11$ then $\frac{(m-3)(m-4)}{12} = 4 + \frac{2}{3}$. Therefore, $\gamma(\coeng(G)) \geq 5$. Hence, the result follows.
\end{proof}

\begin{thm}
If $G$ is isomorphic to  the group $D_{2^{t+1}m}$ or $Q_{2^{t+1}m}$,  where $t \geq 1$ and $m \geq 3$ is odd, then
\begin{enumerate}
\item $\coeng(G)$ is planar if and only if $t = 1$ and $m = 3$.
\item $\coeng(G)$ is toroidal if and only if $t = 1$ and $m = 5, 7$.
\item $\coeng(G)$ is not double-toroidal.
\item $\coeng(G)$ is triple-toroidal if and only if $t = 1$, $m = 9$ and $t = 2$, $m = 3$.
\item  $\gamma(\coeng(G)) \geq 5$ for any other values of $t$ and $m$.
\end{enumerate} 
\end{thm}
\begin{proof}
If $G$ is isomorphic to  $D_{2^{t+1}m}$ or $Q_{2^{t+1}m}$, where $t \geq 1$, $m \geq 3$ and $m$ is odd, then
by Theorem \ref{dihed-genus} we have
\[
\gamma(\coeng(G))  = \frac{m(m - 1)(2^{t-1}-1)^2}{2} + \left\lceil \frac{(m-3)(m-4)}{12}\right
\rceil.
\]
We consider the following cases.

\noindent \textbf{Case 1.} $t = 1$ 	

In this case  we have $(2^{t-1}-1)^2 = 0$  and so  $\frac{m(m - 1)(2^{t-1}-1)^2}{2} = 0$. If  $m = 3$ then  $(m - 3)(m - 4) =0$. Therefore,  $\gamma(\coeng(G)) = 0$. If $m = 5, 7$ then $\left\lceil \frac{(m-3)(m-4)}{12}\right\rceil = 1$.    Therefore,  $\gamma(\coeng(G)) = 1$. If $m = 9$ then $\left\lceil \frac{(m-3)(m-4)}{12}\right\rceil = 3$.    Therefore,  $\gamma(\coeng(G)) = 3$. If $m \geq 11$ then $\left\lceil \frac{(m-3)(m-4)}{12}\right\rceil \geq 5$ and so  $\gamma(\coeng(G)) \geq 5$.

\noindent \textbf{Case 2.} $t = 2$ 	

In this case  we have $(2^{t-1}-1)^2 = 1$. If $m = 3$ then $\frac{m(m - 1)(2^{t-1}-1)^2}{2} = 3$ and $(m - 3)(m - 4) =0$. Therefore,  $\gamma(\coeng(G)) = 3$. If $ m \geq 5$ then $\frac{m(m - 1)(2^{t-1}-1)^2}{2} \geq 10$ and  $\left\lceil \frac{(m-3)(m-4)}{12}\right\rceil \geq 1$. Therefore, $\gamma(\coeng(G)) \geq 11$.

\noindent \textbf{Case 3.} $t \geq 3$ 	

In this case $(2^{t-1}-1)^2 \geq 9$ and $m(m-1) \geq 6$. Therefore, $\frac{m(m - 1)(2^{t-1}-1)^2}{2} \geq 27$. Hence, $\gamma(\coeng(G)) \geq 27$ since  $\left\lceil \frac{(m-3)(m-4)}{12}\right\rceil \geq 0$.
Thus we get the required result.
\end{proof}

\begin{thm}
If $G$ is isomorphic to the group $F_{p, q}$ then 
\begin{enumerate}
\item $\coeng(G)$ is planar if and only if $p = 2$ and $q = 3$.
\item $\coeng(G)$ is toroidal if and only if $p = 2$ and $q = 5, 7$; $p = 3$ and $q = 7$.
\item $\gamma(\coeng(G)) \geq 5$ for any other values of $p$ and $q$. Thus, $\coeng(G)$ is neither double-toroidal nor  triple-toroidal.
\end{enumerate} 
\end{thm}
\begin{proof}
We consider the following cases.

\noindent \textbf{Case 1.} $p = 2$ 	

By Theorem \ref{pq-genus} we have	$\gamma(\coeng(G)) = \left\lceil \frac{(q-3)(q-4)}{12}\right\rceil$. If $q = 3$ then  $\frac{(q-3)(q-4)}{12} = 0$, Therefore, $\gamma(\coeng(G)) = 0$. If $q = 5$ and $7$ then $\frac{(q-3)(q-4)}{12} = \frac{1}{6}$ and $1$ respectively. Therefore, $\gamma(\coeng(G)) = 1$. If $q \geq 11$ then $\frac{(q-3)(q-4)}{12} \geq 4 + \frac{2}{3}$. Therefore, $\gamma(\coeng(G))\geq 5$. 

\noindent \textbf{Case 2.} $p = 3$ 	

By Theorem \ref{pq-genus} we have	
\[
\gamma(\coeng(G)) = \frac{q(q - 1)}{2}\left\lceil\frac{(p-3)^2}{4}\right\rceil + \left\lceil \frac{(q-3)(q-4)}{12}\right\rceil.
\]
In this case we have $\frac{q(q - 1)}{2}\left\lceil\frac{(p-3)^2}{4}\right
\rceil = 0$.  If $q = 7$ then $\frac{(q-3)(q-4)}{12} = 1$. Therefore, $\left\lceil \frac{(q-3)(q-4)}{12}\right\rceil = 1$ and so $\gamma(\coeng(G)) = 1$.  If $q \geq 13$ then $\frac{(q-3)(q-4)}{12} \geq 7.5$. Therefore, $\left\lceil \frac{(q-3)(q-4)}{12}\right\rceil \geq 8$ and so $\gamma(\coeng(G)) \geq 8$.

\noindent \textbf{Case 3.} $p \geq 5$ 	

By Theorem \ref{pq-genus} we have	
\[
\gamma(\coeng(G)) = \frac{q(q - 1)}{2}\left\lceil\frac{(p-3)^2}{4}\right\rceil + \left\lceil \frac{(q-3)(q-4)}{12}\right\rceil.
\]

In this case we have  $\left\lceil\frac{(p-3)^2}{4}\right
\rceil \geq 1$. If $q \geq 11$ then $\frac{q(q - 1)}{2} \geq 55$ and $\left\lceil \frac{(q-3)(q-4)}{12}\right\rceil \geq 5$. Therefore, $\gamma(\coeng(G)) \geq 60$. Hence, the result follows.
\end{proof}

\begin{thm} \label{gen} 
Let $G$ be a finite non-Engel group such that $\omega(\coeng(G))\leq 4$. Then
\begin{enumerate}
\item $\coeng(G)$ is toroidal  if and only if $G\cong A_4, C_3 \times D_{6}$,
\item $\coeng(G)$ is not double-toroidal. 
\end{enumerate}
\end{thm}
\begin{proof} 
Since $\omega(\coeng(G))\leq 4$, by  \cite[Proposition 1.4]{aa1} and  \cite[Theorem 1.2]{aa}, we have $\frac{G}{Z_\infty(G)}\cong D_6$ or $A_4$.

Suppose $\frac{G}{Z_\infty(G)}\cong D_6$. Let $\bar{G}=\frac{G}{Z_\infty(G)}$ and  $\bar{x}=xZ_\infty(G)$ for every $x\in G$. Since $\bar{G} \cong D_6$, we have $\coeng(\bar{G}) \cong \coeng(D_6)$. Let $\phi: \coeng(\bar{G}) \longrightarrow \coeng(D_6)$ be a graph isomorphism. Let $a,b \in D_6$ such that $D_6=\langle a,b \rangle$ and $a^3=b^2=1$. Then $\{b,ab,a^2b\}$ forms a clique in $\coeng(D_6)$. Therefore $\{\bar{x}_1=\phi^{-1}(b),\bar{x}_2=\phi^{-1}(ab),\bar{x}_3=\phi^{-1}(a^2b)\}$ forms a clique of $\coeng(\bar{G})$. By Theorem~\ref{ad}, every element of $x_iZ_\infty(G)$ is adjacent to every element of $x_jZ_\infty(G)$, for $i,j\in \{0,1,2\}, i\neq j$. Thus $K_{m,m,m}$ is a subgraph of $\coeng(G)$, where $m=|Z_\infty(G)|$. Therefore, by \eqref{genus-Kn,n,n}, we have 
\begin{equation}\label{new-eq-2}
\gamma(K_{m,m,m})= \frac{(m-1)(m-2)}{2}\leq \gamma(\coeng(G)). 
\end{equation}
Suppose that $\gamma(\coeng(G))\leq 2$. In that case, by \eqref{new-eq-2}, we get $m\leq 3$. 

If $m=1$, then $G$ is a non-nilpotent group of order $6$. Thus $G\cong D_6$ and $\coeng(D_6)$ is planar.

If $m=2$, then $G$ is a non-nilpotent group of order $12$ with $|Z_\infty(G)|=2$. Thus $G\cong D_{12},Q_{12}$. 
Therefore,  $\coeng(G)$ is planar. 

If $m=3$, then $G$ is a non-nilpotent graph of order $18$ with $|Z_\infty(G)|=3$. Thus $G\cong C_3\times D_6$. Here $\coeng(C_3\times D_6) \cong K_{3,3,3}$ and $\gamma(K_{3,3,3})=\frac{(3-1)(3-2)}{2}=1$. Thus $\coeng(C_3\times D_6)$ is toroidal. An embedding of $\coeng(C_3\times D_6)$ on a torus is shown in Figure \ref{fig1}, where $C_3=\langle a\rangle$.

\begin{figure}[htbp]
\begin{center}
\setlength{\unitlength}{3pt}
\begin{picture}(50,30)(0,0)

\put(-25,-20){\dashbox{1}(100,45)[br]}

\put(-25,10){\circle*{3}}
\put(-25,-5){\circle*{3}}

\put(75,-5){\circle*{3}}
\put(75,10){\circle*{3}}

\put(0,25){\circle*{3}}
\put(25,25){\circle*{3}}
\put(50,25){\circle*{3}}

\put(0,-20){\circle*{3}}
\put(20,-20){\circle*{3}}
\put(50,-20){\circle*{3}}

\put(8.3,10){\circle*{3}}
\put(41.6,10){\circle*{3}}
\put(8.3,-5){\circle*{3}}
\put(41.6,-5){\circle*{3}}

\put(75,-7){\line(0,1){16}}
\put(-25,10){\line(0,-1){16}}

\put(0,25){\line(1,0){25}}
\put(25,25){\line(1,0){25}}
\put(50,25){\line(1,0){25}}
\put(-25,25){\line(1,0){25}}

\put(0,-20){\line(1,0){25}}
\put(25,-20){\line(1,0){25}}
\put(50,-20){\line(1,0){25}}
\put(-25,-20){\line(1,0){25}}

\put(-25,10){\line(1,0){33}}
\put(8.3,10){\line(1,0){33}}
\put(41.6,10){\line(1,0){33}}

\put(-25,-5){\line(1,0){33}}
\put(8.3,-5){\line(1,0){33}}
\put(41.6,-5){\line(1,0){33}}

\put(-25,10){\line(5,-2){34}}
\put(8.3,10){\line(5,-2){34}}
\put(41.6,10){\line(5,-2){34}}

\put(8.3,-5){\line(0,1){15}}
\put(41.6,-5){\line(0,1){15}}

\put(0,25){\line(3,-5){9}}
\put(25,25){\line(6,-5){17}}
\put(50,25){\line(5,-3){25}}

\put(0,25){\line(-5,-3){25}}
\put(25,25){\line(-6,-5){17}}
\put(50,25){\line(-3,-5){9}}

\put(-25,-5){\line(5,-3){25}}

\put(-25,-5){\line(3,-1){45}}

\put(8.3,-5){\line(5,-6){13}}
\put(8.3,-5){\line(3,-1){42}}
\put(41.6,-5){\line(3,-5){9}}
\put(41.6,-5){\line(4,-1){33.5}}

\put(0,-20){\line(-4,1){25}}

\put(-42,9){$(1, ab)$}
\put(-42,-6){$(1, a^2b)$}

\put(76,9){$(1, ab)$}
\put(76,-6){$(1, a^2b)$}

\put(-8,27){$(1, b)$}
\put(17,27){$(a, ab)$}
\put(42,27){$(a^2, a^2b)$}

\put(-8,-24){$(1, b)$}
\put(14,-24){$(a, ab)$}
\put(42,-24){$(a^2, a^2b)$}

\put(-9,11){$(a, a^2b)$}
\put(28,11){$(a^2, b)$}
\put(9,-4){$(a, b)$}
\put(42,-4){$(a^2, ab)$}

\end{picture}
\vspace{3cm} \caption{Embedding of $\coeng(C_3\times D_6)$ on a torus.} \label{fig1}
\end{center}
\end{figure}

\begin{figure}[htbp]
	\begin{center}
		\setlength{\unitlength}{3pt} 
		\begin{picture}(50,50)(0,0)

			\put(0,40){\circle*{3}}
			\put(40,40){\circle*{3}}
			\put(70,17){\circle*{3}}
			\put(70,-3){\circle*{3}}
			\put(40,-25){\circle*{3}}
			\put(0,-25){\circle*{3}}
			\put(-30,-3){\circle*{3}}
			\put(-30,17){\circle*{3}}

			\put(0,40){\line(1,0){40}}
			\put(0,40){\line(3,-1){70}}
			\put(0,40){\line(5,-3){71}}
			\put(0,40){\line(-4,-3){30}}
			\put(0,40){\line(0,-1){65}}
			\put(0,40){\line(3,-5){39}}
			
			\put(40,40){\line(4,-3){30}}
			\put(40,40){\line(0,-1){65}}
			\put(40,40){\line(-5,-3){71}}
			\put(40,40){\line(-3,-1){70}}
			\put(40,40){\line(-3,-5){39}}

			\put(70,17){\line(0,-1){20}}
			\put(70,17){\line(-5,-3){70}}
			\put(70,17){\line(-5,-1){100}}
			\put(70,17){\line(-1,0){100}}

			\put(70,-4){\line(-3,-2){31}}
			\put(70,-4){\line(-1,0){100}}
			\put(70,-2.5){\line(-3,-1){70}}
			\put(70,-4){\line(-5,1){100}}
			
			\put(40,-25){\line(-1,0){40}}
			\put(40,-25){\line(-5,3){70}}
			\put(40,-25){\line(-3,1){70}}
			
			\put(0,-25){\line(-4,3){31}}
			
			\put(-30,-4){\line(0,1){22}}

			\put(-6,42){$(2,3,4)$}
			\put(34,42){$(1,2,3)$}
			\put(72,16){$(1,3,4)$}
			\put(72,-5){$(1,3,2)$}
			\put(34,-29){$(1,4,3)$}
			\put(-6,-29){$(1,4,2)$}
			\put(-44,-5){$(2,4,3)$}
			\put(-44,16){$(1,2,4)$}
			
		\end{picture}
		\vspace{3cm} \caption{$\coeng(A_4)$} \label{figA4}
	\end{center}
\end{figure}

\begin{figure}[htbp]
\begin{center}
\setlength{\unitlength}{3pt}
\begin{picture}(50,40)(0,0)

\put(-23,-25){\dashbox{1}(95,60)[br]}

\put(7,35){\circle*{3}}
\put(42,35){\circle*{3}}

\put(72,21){\circle*{3}}
\put(72,-11){\circle*{3}}

\put(42,-25){\circle*{3}}
\put(7,-25){\circle*{3}}

\put(-23,-11){\circle*{3}}
\put(-23,21){\circle*{3}}

\put(7,22){\circle*{3}}
\put(7,-1){\circle*{3}}

\put(42,-12){\circle*{3}}
\put(42,0){\circle*{3}}

\put(72,21){\line(0,-1){32}}
\put(72,21){\line(-2,1){30}}
\put(72,21){\line(-1,0){66}}
\put(72,21){\line(-5,1){66}}

\put(-23,21){\line(3,-1){66}}
\put(-23,21){\line(3,-2){31}}
\put(-23,21){\line(0,-1){31}}

\put(72,-11){\line(-1,0){30}}
\put(72,-11){\line(-2,1){66}}
\put(72,-11){\line(-3,1){31}}

\put(-23,-11){\line(3,1){31}}
\put(-23,-11){\line(2,-1){30}}

\put(7,35){\line(1,0){34}}
\put(7,35){\line(0,-1){14}}

\put(7,-25){\line(0,1){25}}
\put(7,-25){\line(1,0){35}}
\put(7,-25){\line(-5,1){30}}
\put(72,-19){\line(-5,1){30}}

\put(7,22){\line(-5,1){30}}
\put(72,28){\line(-4,1){30}}

\put(42,-25){\line(0,1){12}}
\put(42,-25){\line(-3,2){36}}

\put(7,22){\line(5,-3){36}}
\put(7,22){\line(-3,1){30}}
\put(72,32){\line(-3,1){10}}
\put(64,-25){\line(-5,3){22}}

\put(42,0){\line(0,-1){25}}
\put(42,0){\line(-1,0){35}}
\put(42,0){\line(-5,2){65}}
\put(72,25){\line(-3,1){30}}

\put(42,-12){\line(-3,1){35}}

\put(73,20){$(2,3,4)$}
\put(-37,20){$(2,3,4)$}

\put(73,-12){$(1,2,3)$}
\put(-37,-12){$(1,2,3)$}

\put(1,37){$(1,3,4)$}
\put(1,-30){$(1,3,4)$}

\put(36,37){$(1,3,2)$}
\put(36,-30){$(1,3,2)$}

\put(9,23){$(1,4,2)$}
\put(-8,-1){$(1,2,4)$}

\put(28,-15){$(2,4,3)$}
\put(29,-4){$(1,4,3)$}

\put(-26,-20){$a$}
\put(73,-20){$a$}

\put(-26,24){$b$}
\put(73,24){$b$}

\put(-26,27){$c$}
\put(73,27){$c$}

\put(-26,31){$d$}
\put(73,31){$d$}

\put(61,36){$e$}
\put(63,-28){$e$}

\end{picture}
\vspace{3cm} \caption{Embedding of $\coeng(A_4)$ on a torus.} \label{figeA4}
\end{center}
\end{figure} 

Suppose $\frac{G}{Z_\infty(G)}\cong A_4$. Suppose that $|Z_\infty(G)|\geq 2$. Let $\bar{G}=\frac{G}{Z_\infty(G)}$ and $\bar{x}=xZ_\infty(G)$ for every $x\in G$. Since $\bar{G} \cong A_4$, we have $\coeng(\bar{G}) \cong \coeng(A_4)$. Let $\phi: \coeng(\bar{G}) \longrightarrow \coeng(A_4)$ be a graph isomorphism. As we see in Figure \ref{figA4}, $\coeng(A_4)$ has a subgraph isomorphic to $K_{4,4}$ with parts $H=\{(2,3,4),(1,2,4),(2,4,3),(1,4,2)\}$ and $K=\{(1,2,3),(1,3,4),(1,3,2),(1,4,3)\}$. Therefore, $\coeng(\bar{G})$ has a subgraph isomorphic to $K_{4,4}$ with parts 
$$\bar{H}=\{\phi^{-1}(2,3,4),\phi^{-1}(1,2,4),\phi^{-1}(2,4,3),\phi^{-1}(1,4,2)\}$$ and 
$$\bar{K}=\{\phi^{-1}(1,2,3),\phi^{-1}(1,3,4),\phi^{-1}(1,3,2),\phi^{-1}(1,4,3)\}.$$

Let $z\in Z_\infty(G),z\neq 1$. Then, by Theorem~\ref{ad}, every element of $\bar{H}\cup z\bar{H}$ is adjacent to every element of $\bar{K}\cup z\bar{K}$, showing that $K_{8,8}$ is a subgraph of $\coeng(G)$. 

Thus $\gamma(\coeng(G))\geq \gamma(K_{8,8})=9$, a contradiction. Thus $|Z_\infty(G)|=1$ and so $G\cong A_4$. An embedding of the  graph $\coeng(A_4)$ on a torus is shown in Figure \ref{figeA4}. This completes the proof.
\end{proof}

Let $N_k$ denote the surface formed by a connected sum of $k$ projective planes ($k\geq 1$). The number $k$ is called the crosscap of $N_k$. A simple graph $\Gamma$ which can be embedded in $N_k$ but not in $N_{k-1}$, is call a graph of crosscap $\bar{\gamma}(\Gamma) = k$.  
A graph $\Gamma$ with $\bar{\gamma}(\Gamma)=1$ is called a projective graph. The following theorem is useful in the next result. 

\begin{thm} {\rm (\cite{bou, rin}). }
For positive integers $m$ and $n$, we have 
\begin{enumerate}
\item $\bar{\gamma}(K_{m,n})=\lceil \frac{(m-2)(n-2)}{2}  \rceil$ if $m,n \geq 2$.
		
\item $ \bar{\gamma}(K_n)= \begin{cases} 
			\lceil \frac{(n-3)(n-4)}{6} \rceil, & \textrm{ if $n\geq 3$ and $n\neq 7$,} \\
			3, & \textrm{ if $n = 7$.} \\
		\end{cases} $
\end{enumerate}
\end{thm}

We conclude this section by determining all finite non-Engel groups $G$ (up to isomorphism) such that $\omega(\coeng(G))\leq 4$ and $\coeng(G)$ is projective.
\begin{thm} 
Let $G$ be a finite non-Engel group such that $\omega(\coeng(G))\leq 4$. Then $\coeng(G)$ is projective if and only if $G \cong D_6, D_{12},Q_{12}$.
\end{thm}
\begin{proof}
 By \cite[Proposition 1.4]{aa1} and \cite[Theorem 1.2]{aa}, we have $G/Z_\infty(G)\cong D_6$ or $A_4$. Suppose $G/Z_\infty(G)\cong A_4$. As seen in Figure \ref{figA4}, $\coeng(A_4)$ has a subgraph isomorphic to $K_{4,4}$. Thus $\coeng(G/Z_\infty(G))$ has a subgraph isomorphic to $K_{4,4}$. Note that $|v(\coeng(G/Z_\infty(G)))|=|v(\coeng(A_4))|=8$. Let $v(\coeng(G/Z_\infty(G)))=\{\bar{a_1},\bar{a_2},\dots,\bar{a_8}\}$. Then the induced subgraph of $\coeng(G)$ by the set $\{a_1,a_2,\dots,a_8\}$ has a subgraph isomorphic to $K_{4,4}$. Thus $\bar{\gamma}(\coeng(G))\geq \bar{\gamma}(K_{4,4})=2$, a contradiction. Thus $G/Z_\infty(G)\cong D_6$. 

Let $m=|Z_\infty(G)|$. As seen in the proof of Theorem \ref{gen}, $K_{m,m,m}$ is a subgraph of $\coeng(G)$. Thus if $m\geq 3$, then $K_{6,3}$ is a subgraph of $\coeng(G)$. But $\bar{\gamma}(K_{6,3})=2$. Therefore, $m\leq 2$. If $m=1$, then $G\cong D_6$ and if $m=2$, then $G\cong D_{12}, Q_{12}$. Note that $\coeng(D_6) \cong K_3$ 
is projective. Also, from Theorem \ref{dihed}, $\coeng(D_{12})\cong \coeng(Q_{12})\cong K_{3 \cdot 2}$
which is projective. This completes the proof.
\end{proof}

\section{Various spectra and energies of $\coeng(G)$}\label{Energy}

Let $A(\Gamma)$ and $D(\Gamma)$ be the adjacency matrix and degree matrix of a finite simple undirected graph $\Gamma$. Let $L(\Gamma) := D(\Gamma) - A(\Gamma)$ and $Q(\Gamma) := D(\Gamma) + A(\Gamma)$ be the Laplacian and signless Laplacian matrices of $\Gamma$. The characteristic polynomials of $A(\Gamma), L(\Gamma)$ and $Q(\Gamma)$ are called characteristic polynomial (denoted by $P(\Gamma, x)$), Laplacian polynomial (denoted by $P_L(\Gamma, x)$) and signless Laplacian polynomial (denoted by $P_Q(\Gamma, x)$) of $\Gamma$ respectively. The set of all the roots (with multiplicities) of $P(\Gamma, x)$,  $P_L(\Gamma, x)$ and  $P_Q(\Gamma, x)$  are called spectrum, Laplacian spectrum and signless Laplacian spectrum of $\Gamma$ denoted by $\Spec(\Gamma), \L-Spec(\Gamma)$ and $\Q-Spec(\Gamma)$ respectively. If $\{(\alpha_1)^{n_1}, (\alpha_2)^{n_2}, \dots, (\alpha_k)^{n_k}\}$ represents $\Spec(\Gamma), \L-Spec(\Gamma)$ or $\Q-Spec(\Gamma)$ then $\alpha_i$'s are roots of $P(\Gamma, x)$, $P_L(\Gamma, x)$ or $P_Q(\Gamma, x)$ with multiplicities $n_i$. A graph $\Gamma$ is called A-integral, L-integral and Q-integral if  $\Spec(\Gamma), \L-Spec(\Gamma)$ and $\Q-Spec(\Gamma)$ contain only integers. If  $\Gamma$ is   A-integral, L-integral and Q-integral simultaneously then it is called ALQ-integral.

The energy,  Laplacian energy and signless Laplacian energy of $\Gamma$ denoted by $E(\Gamma)$, $LE(\Gamma)$ and  $LE^{+}(\Gamma)$ respectively are defined as
\[
E(\Gamma) = \sum_{\lambda \in \Spec(\Gamma)} |\lambda|,
\]
\[
LE(\Gamma)=\sum_{\mu \in \L-Spec(\Gamma)} \left| \mu -\frac{2|e(\Gamma)|}{|v(\Gamma)|} \right| \quad \text{and} \quad LE^{+}(\Gamma)= \sum_{\nu \in \Q-Spec(\Gamma)} \left| \nu-\frac{2|e(\Gamma)|}{|v(\Gamma)|} \right|,
\]
where $v(\Gamma)$ and $e(\Gamma)$ are the sets of vertices and edges of $\Gamma$ respectively. 
A graph is called hyperenergetic or hypoenergetic if $E(\Gamma) > E(K_{|v(\Gamma)|})$ or $E(\Gamma) < |v(\Gamma)|$ respectively. It is not known whether co-Engel graphs of finite groups are hyperenergetic or hypoenergetic in general. However, we shall show that the $\coeng(G)$ for finite groups considered in this paper are neither hyperenergetic nor hypoenergetic. In \cite{E-LE-Gutman}, Gutman et al. conjectured that $E(\Gamma) \leq LE(\Gamma)$ for any finite graph $\Gamma$ (which is known as the E-LE conjecture). In general, this conjecture is false. It is not known whether co-Engel graphs of finite groups satisfy the E-LE conjecture. Our computations show that the co-Engel graphs of finite groups considered in this paper satisfy this conjecture. 

The following two well-known results are useful in  computing various spectra and energies of $\coeng(G)$ for the groups considered in Section~\ref{s:next}.
\begin{lem} \label{Kn-poly}
If  $\Gamma = K_{n}$ then
	\[
	P(\Gamma, x)= (x+1)^{n-1}\left(x- (n - 1)\right), \quad
	P_L(\Gamma, x)= x(x - n)^{n - 1} \text{ and }
	\]
	\[
	P_Q(\Gamma, x)=  (x- (n -2))^{n-1}(x-2(n - 1)).
	\]	
\end{lem}

\begin{lem}\label{Kab-poly}
If  $\Gamma = K_{a \cdot b}$ then
\[
P(\Gamma, x)=x^{a(b-1)}(x+b)^{a-1}\left(x- b(a - 1)\right), 
\]
\[
P_L(\Gamma, x)= x(x - b(a-1))^{a(b-1)} (x - ab)^{a - 1} \text{ and }
\]
\[
P_Q(\Gamma, x)=  (x- b(a - 1))^{a(b-1)}(x- b(a -2))^{a-1}(x-2b(a - 1)).
\]	
\end{lem}

\noindent If  $m \geq 3$ is odd and $G \cong D_{2m}$ then, by Theorem \ref{dihede},	we have $\coeng(G) \cong K_m$. Therefore, using Lemma \ref{Kn-poly}, we get the following result.
\begin{thm}
If $G$ is isomorphic to $D_{2m}$,  where  $m \geq 3$ is odd, then 
\[
\Spec(\coeng(G)) = \left\{(-1)^{m - 1}, (m - 1)^1\right\}, \quad 
\L-Spec(\coeng(G)) = \left\{ (0)^1, m^{m - 1}\right\},
\]
\[
\Q-Spec(\coeng(G))\!=\!\left\{(m - 2)^{m - 1}, (2(m - 1))^1\right\}\!\text{ and }
\]
\[
 E(\coeng(G))\!=\!LE(\coeng(G))\!=\!LE^+(\coeng(G)) = 2(m - 1). 
\]
\end{thm}
Using Theorem \ref{dihed}, Theorem \ref{pq} and Lemma \ref{Kab-poly}, we get the following results. 
\begin{thm}\label{Energy-D2m}
If $G$ is isomorphic to   $D_{2^{t+1}m}$ or  $Q_{2^{t+1}m}$, where $t \geq 1$ and $m \geq 3$ is odd, then 
\[
\Spec(\coeng(G)) = \left\{ (0)^{m(2^t-1)}, (-2^t)^{m-1}, \left(2^t(m - 1)\right)^1\right\}, 
\]
\[
\L-Spec(\coeng(G)) = \left\{(0)^1, (2^t(m-1))^{m(2^t-1)}, (2^tm)^{m - 1} \right\},  
\]
\[
\Q-Spec(\coeng(G)) = \left\{(2^t(m - 1))^{m(2^t-1)}, (2^t(m -2))^{m-1}, (2^{t+1}(m - 1))^1\right\}
\]
 and  $E(\coeng(G)) = LE(\coeng(G)) = LE^+(\coeng(G)) = 2^{t+1}(m - 1)$. 
\end{thm}

\begin{thm}\label{Energy-Fpq}
	If $G$ is isomorphic to   $F_{p, q}$, then 
	\[
	\Spec(\coeng(G)) = \left\{ (0)^{q(p-2)}, (-(p-1))^{q-1}, \left((p-1)(q - 1)\right)^1\right\} , 
	\]
	\[
	\L-Spec(\coeng(G)) = \left\{(0)^1, ((p-1)(q-1))^{q(p-2)}, (q(p-1))^{q - 1} \right\} ,  
	\]
	\[
	\Q-Spec(\coeng(G))\!=\!\left\{\!((p\!-\!1)(q - 1))^{q(p-2)}, ((p-1)(q -2))^{q-1}, (2(p-1)(q - 1))^1\!\right\} ,
	\]
	\[
	\text{ and } E(\coeng(G)) = LE(\coeng(G)) = LE^+(\coeng(G)) = 2(p-1)(q-1) .
	\]
	
\end{thm}

 We conclude this section with the following two corollaries.
\begin{cor}
	If $G$ is isomorphic to   $D_{2^{t+1}m}$ or  $Q_{2^{t+1}m}$, where $t \geq 1$ and $m \geq 3$ is odd, then 
	$\coeng(G)$ is neither hyperenergetic nor hypoenergetic.
\end{cor}

\begin{proof}
	We have $|v(\coeng(G))| = 2^tm$ and 
	$
	E(K_{2^tm}) = 2(2^tm - 1).
	$
	By Theorem \ref{Energy-D2m} we get
	\[
	E(\coeng(G)) - E(K_{|v(\coeng(G))|}) = 2^{t+1}(m-1) - 2(m2^t - 1)
	= -2(2^t - 1) < 0. 
	\]
	Therefore, $\coeng(G)$ is not hyperenergetic. We also have
	\[
	E(\coeng(G)) - |v(\coeng(G))| = 2^{t+1}(m-1) - 2^tm
	= 2^t(m - 2) > 0.
	\]
	Hence, $\coeng(G)$ is not hypoenergetic.
\end{proof}

\begin{cor}
If $G$ is isomorphic to   $F_{p, q}$, then 
 $\coeng(G)$ is neither hyperenergetic nor hypoenergetic.
\end{cor}

\begin{proof}
We have $|v(\coeng(G))| = q(p-1)$ and
$
E(K_{q(p-1)}) = 2(q(p-1)-1).
$
By Theorem \ref{Energy-Fpq} we get
\[
E(\coeng(G)) - E(K_{|v(\coeng(G))|}) = 2(p - 1)(q-1) - 2(q(p-1)-1)
	= -2(p - 2) \leq 0. 
\]
Therefore, $\coeng(G)$ is not hyperenergetic. We also have
\[
E(\coeng(G)) - |v(\coeng(G))| = 2(p - 1)(q-1) - q(p - 1) 
	= (p - 1)(q - 2) > 0.
\]
Hence, $\coeng(G)$ is not hypoenergetic.
\end{proof}

\section{Zagreb indices of $\coeng(G)$}
In this section, we compute Zagreb indices of $\coeng(G)$ for certain dihedral groups, dicyclic groups and the group $F_{p, q}$ and check whether they satisfy {the Hansen-Vuki{\v{c}}evi{\'c} conjecture.  The first and second Zagreb indices of a simple undirected graph $\Gamma$, denoted by $M_{1}(\Gamma)$ and $M_{2}(\Gamma)$ respectively, are given by 
\[
M_{1}(\Gamma) = \sum\limits_{v \in v(\Gamma)} \deg(v)^{2}  \text{ and }  M_{2}(\Gamma) = \sum\limits_{uv \in e(\Gamma)} \deg(u)\deg(v),
\]
where $\deg(v)$ is  the  degree of $v$. The following lemma is useful in obtaining our results.
\begin{lem}\label{ZI-lem}
If $\Gamma$ is isomorphic to the graph $K_{a \cdot b}$ then $M_{1}(\Gamma) = a(a - 1)^2b^3 $ and  $M_{2}(\Gamma) = \frac{a(a - 1)^{3}b^{4}}{2} $.
\end{lem}

 In 2007,  Hansen and Vuki{\v{c}}evi{\'c} \cite{hansen2007comparing}  conjectured that
\begin{equation}\label{Conj}
		\dfrac{M_{2}(\Gamma)}{\vert e(\Gamma) \vert} \geq \dfrac{M_{1}(\Gamma)}{\vert v(\Gamma) \vert} .
\end{equation}
In general, \eqref{Conj} is not true. However, it is not known whether \eqref{Conj} satisfies for $\coeng(G)$. The following results show that \eqref{Conj} is  true for $\coeng(G)$ if $G$ is isomorphic to $D_{2^{t+1}m}$,  $Q_{2^{t+1}m}$ and $F_{p, q}$, where $t \geq 1$,  $m \geq 3$ is odd and $p$, $q$ are primes such that $q\equiv 1 \mod p$.
\begin{thm}
If $G$ is isomorphic to   $D_{2^{t+1}m}$ or  $Q_{2^{t+1}m}$, where $t \geq 1$ and $m \geq 3$ is odd, then 
\[
M_1(\coeng(G)) = 2^{3t}m(m - 1)^2 \text{ and } M_2(\coeng(G)) = 2^{4t - 1}m(m - 1)^3
\]
Further, $\frac{M_{2}(\coeng(G))}{\vert e(\coeng(G)) \vert} = 2^{2t}(m - 1)^2 = \frac{M_{1}(\coeng(G))}{\vert v(\coeng(G)) \vert}$.
\end{thm}

\begin{proof}
By Theorem \ref{dihed}	we have $\coeng(G) \cong K_{m\cdot 2^t}$. Therefore, by Lemma \ref{ZI-lem}, we get the expressions for  $M_1(\coeng(G))$ and $M_2(\coeng(G))$.
Here $|v(\coeng(G))| = 2^tm$ and  $|e(\coeng(G))| = 2^{2t - 1}m(m - 1)$. Therefore
\[
\frac{M_{1}(\coeng(G))}{\vert v(\coeng(G)) \vert} = \frac{2^{3t}m(m - 1)^{2}}{2^{t}m} = 2^{2t}(m - 1)^2 = \frac{2^{4t - 1}m(m - 1)^{3}}{2^{2t - 1}m(m - 1} = \frac{M_{2}(\coeng(G))}{\vert e(\coeng(G)) \vert}.
\]
\end{proof}
Similarly, using Theorem \ref{pq} and Lemma \ref{ZI-lem}, we get the following result.
\begin{thm}
If $G$ is isomorphic to   $F_{p, q}$, then 
\[
M_1(\coeng(G)) = q(q - 1)^2(p - 1)^3 \text{ and } M_2(\coeng(G)) = \frac{q(q - 1)^{3}(p - 1)^{4}}{2}
\]
Further, $\displaystyle{\frac{M_{2}(\coeng(G))}{\vert e(\coeng(G)) \vert} = (q - 1)^2(p - 1)^2 = \frac{M_{1}(\coeng(G))}{\vert v(\coeng(G)) \vert}}$.
\end{thm}

\section{Open problems}

\begin{prob}
The arc sets of the $k$-Engel digraphs form an increasing sequence under
inclusion. In a finite group, the union of the arc sets of the $k$-Engel
digraphs forms the Engel digraph studied in this paper. Similarly, the edge
sets of the $k$-nilpotency graphs form an increasing sequence. The
undirected $k$-Engel graph contains the $k$-nilpotency graph.

Study the relationship between these graphs and digraphs.
\end{prob}

\begin{prob}
What can be said about the set of positive integers $n$ for which there are
groups of order $n$ with isomorphic Engel graphs but non-isomorphic Engel
digraphs? We conjecture that this set contains $2p^3$ for every odd prime $p$.
\end{prob}

\begin{prob}
Describe the classes of finite groups which are one-sided or two-sided Engel
tame, by excluded subgroups or otherwise.
\end{prob}

\begin{prob}
Study the strong Engel graph. As we observed in Section~\ref{s:NFE},
a group is nilpotent if and only if its Engel graph and strong Engel graph are
equal, and is two-sided Engel tame if and only if its strong Engel graph is
equal to its nilpotency graph. If these conditions are not satisfied, then we
have two ``difference graphs'' to study.
\end{prob}

The characterization of finite non-Engel groups such that $\coeng(G)$ is planar is given in Theorem \ref{Ab-Plannar}. In Section \ref{S:Genus}, we characterized certain finite non-Engel groups such that $\coeng(G)$ is toroidal, double-toroidal or triple-toroidal. In this regard, we pose the following problem.
\begin{prob}
Characterize finite  non-Engel groups such that $\coeng(G)$ is  toroidal, double-toroidal or triple-toroidal.   
\end{prob}
In Section \ref{Energy}, it was observed that $\coeng(G)$ is neither hyperenergetic nor hypoenergetic and it satisfies E-LE conjecture if $G = D_{2m}, D_{2^{t+1}m}, Q_{2^{t+1}m}$ and $F_{p, q}$, where $t \geq 1$, $m \geq 3$ is odd and $p, q$ are primes. Further, $\coeng(G)$ satisfies Hansen-Vuki{\v{c}}evi{\'c} conjecture for these groups. In this regard we pose the following problem.
\begin{prob}
Determine whether (a) $\coeng(G)$  satisfies E-LE conjecture (b) $\coeng(G)$ is neither hyperenergetic nor hypoenergetic        
(c) $\coeng(G)$  satisfies Hansen-Vuki{\v{c}}evi{\'c} conjecture. 
\end{prob}

\end{document}